\documentclass[11pt, oneside]{amsart}   	
\usepackage[left=2.5cm,top=2.5cm,right=2.5cm,bottom=2.5cm]{geometry}
\usepackage{graphicx}				
\usepackage{amssymb, amsmath, amsthm}

\usepackage{todonotes}

\newtheorem{theorem}{Theorem}
\newtheorem{lemma}[theorem]{Lemma}
\newtheorem{proposition}[theorem]{Proposition}
\newtheorem{conj}[theorem]{Conjecture}
\newtheorem{corollary}[theorem]{Corollary}
\newtheorem*{nota}{Notation}
\newtheorem*{claim}{Claim}
\newtheorem*{example}{Example}
\newtheorem{rmk}[theorem]{Remark}

\newtheorem{definition}[theorem]{Definition}

\newcommand{\Z}{ \mathbb{Z}}
\newcommand{\F}{ \mathbb{F}}

\newcommand{\lcm}{ {\rm lcm}}

\newcommand{\gcmd}{{\rm gcmd}}

\newcommand{\jac}[2]{ \left( #1 \mid #2 \right)}
\newcommand{\jacs}[2]{ \left( #1 \mid #2 \right)}

\newcommand{\floor}[1]{\left\lfloor #1 \right\rfloor}
\providecommand{\abs}[1]{\left\vert #1 \right\vert}

\title{Average liar count for degree-$2$ Frobenius pseudoprimes}

\author[A. Fiori]{Andrew Fiori}
\thanks{A.F. gratefully acknowledges support from the Pacific Institute for Mathematical Sciences (PIMS)}
\address{%
Mathematics and Computer Science,
C526 University Hall, 4401 University Drive,
University of Lethbridge,
Lethbridge, Alberta, T1K 3M4}
\email{andrew.fiori@uleth.ca}

\author[A. Shallue]{Andrew Shallue}
\thanks{A.S. supported by an Artistic and Scholarly Development grant from Illinois Wesleyan University}
\address{Illinois Wesleyan University, 1312 Park St, Bloomington, IL, 61701 USA}
\email{ashallue@iwu.edu}

\begin{document}

\begin{abstract}
In this paper we obtain lower and upper bounds on the average number of liars for the Quadratic Frobenius Pseudoprime 
Test of Grantham \cite{Grantham01}, generalizing arguments of Erd\H{o}s and Pomerance \cite{ErdosPomerance01} and 
Monier \cite{Monier01}.  These bounds are provided for both Jacobi symbol $\pm 1$ cases, providing evidence for 
the existence of several challenge pseudoprimes.
\end{abstract}



\maketitle

\section{Introduction}

A pseudoprime is a composite number that satisfies some necessary condition for primality.
Since primes are necessary building blocks for so many algorithms, and since the most common way 
to find primes in practice is to apply primality testing algorithms based on such necessary conditions, 
it is important to gather what information we can about pseudoprimes.  In addition to the practical benefits, 
pseudoprimes have remarkable divisibility properties that make them fascinating objects of study.

The most common necessary condition used in practice is that the number has no small 
divisors.  Another common necessary condition follows from a theorem of Fermat, 
that if $n$ is prime and $\gcd(a,n)=1$ then $a^{n-1} = 1 \pmod{n}$.
If $\gcd(a,n) = 1$ and $a^{n-1} = 1 \pmod{n}$ for composite $n$ we call $a$ a Fermat liar, 
and denote by $F(n)$ the set of Fermat liars with respect to $n$, or more precisely the set of 
their residue classes modulo $n$.

For the purposes of generalization, it is useful to translate the Fermat condition to polynomial rings.
Let $n$ be prime, let $R = \Z/n\Z$, assume $a \in R^{\times}$, and construct the polynomial ring 
$R[x]/\langle x-a \rangle$.  Then a little work shows that $x^n = x$ in $R[x]/\langle x-a \rangle$
\cite[Proof of Theorem 4.1]{Grantham01}.  After all, as $x = a$ in $R[x]/\langle x-a \rangle$, 
we have $R[x]/\langle x-a \rangle \cong R$ as fields, and $a^n = a$ in $R$.  The advantage of 
this view is that $x-a$ may be replaced by an arbitrary polynomial.

In the following definition $\gcmd$ stands for ``greatest common monic divisor," and implicitly depends 
on a modulus $n$.
Following \cite{Grantham01}, for monic polynomials $g_1(x), g_2(x), f(x) \in (\Z/n\Z)[x]$ 
we say that $\gcmd(g_1(x), g_2(x)) = f(x)$ if
 the ideal generated by $g_1(x), g_2(x)$ is principal and equals 
the ideal generated by $f(x)$ in $(\Z/n\Z)[x]$.  

\begin{definition}[\cite{Grantham01}, Section 3]
Let $f(x) \in \Z[x]$ be a monic polynomial of degree $d$ and discriminant $\Delta$.  
Then odd composite $n$ is a Frobenius pseudoprime with respect to $f(x)$ if 
the following conditions all hold.
\begin{enumerate}
\item (Integer Divisibility) We have $\gcd(n, f(0)\Delta) = 1$.

\item (Factorization) Let $f_0(x) = f(x) \pmod{n}$.   Define 
$F_i(x) = \gcmd(x^{n^i}-x, f_{i-1}(x))$ and $f_i(x) = f_{i-1}(x)/F_i(x)$ for $1 \leq i \leq d$.  All of the $\gcmd$s
exist and $f_d(x) = 1$.

\item (Frobenius) For $2 \leq i \leq d$, $F_i(x) \mid F_i(x^n)$.

\item (Jacobi) Let $S = \sum_{2 \mid i} \deg(F_i(x))/i$.  Have $(-1)^S = \jacs{\Delta}{n}$, 
where $ \jacs{\Delta}{n}$ is the Jacobi symbol.
\end{enumerate}

\end{definition}

If $n$ is prime, then $(\Z/n\Z)$ is a field, making $(\Z/n\Z)[x]$ a principal ideal domain, from which 
it follows that $\gcmd(g_1(x), g_2(x))$ must exist.   As an example of non-existence note that 
the ideal $\langle x+2, x \rangle \subseteq (\Z/6\Z)[x]$ is non-principal and thus $\gcmd(x+2, x)$ 
does not exist modulo $6$.
Grantham shows that if $\gcmd(g_1(x), g_2(x))$ exists in $(\Z/n\Z)[x]$ then for each prime $p \mid n$, 
 the usual $\gcd(g_1(x), g_2(x))$, when taken over $\Z/p\Z$, will have the same degree \cite[Corollary 3.3]{Grantham01}.  
 Furthermore, the Euclidean algorithm when applied to $g_1(x), g_2(x)$ will either correctly compute their $\gcmd$, 
 or find a proper factor of $n$ when the leading coefficient of a 
 remainder fails to be a unit \cite[Proposition 3.5]{Grantham01}. 
 Returning to the earlier example, 
 if we apply the Euclidean algorithm to attempt to compute $\gcd(x, x+2)$ modulo $6$, 
 the first recursive step yields $\gcd(x+2, 2)$ and then in performing the resulting division 
 the attempt to invert $2$ fails and yields a factor of $6$.
 
\begin{example}
Suppose $d = 1$ and $n$ is a Frobenius pseudoprime with respect to $f(x) = x-a$.
Then $\gcd(a,n) = 1$ and $\gcmd(x^n-x, x-a) = x-a$, which implies $a$ is a unit and  $a^n = a \pmod{n}$. 
Hence $a$ is a Fermat liar with respect to 
$n$.  Conversely, if $a$ is a Fermat liar then $\gcd(a,n)=1$ and $\gcmd(x^n-x, x-a) = x-a$, from which we conclude that $n$ is a Frobenius 
pseudoprime with respect to $x-a$.
\end{example}

A Frobenius liar is then a polynomial, and to count the liars with respect to $n$ we restrict our polynomials 
to members of $(\Z/n\Z)[x]$.
We denote by $L_d(n)$ the set of Frobenius liars of degree $d$ with respect to $n$, and note by the example above
that $L_1(n) = F(n)$.  We will further divide the set $L_2(n)$ into $L^+_2(n)$ and $L^-_2(n)$.
A degree $2$ polynomial $f(x)$ with discriminant $\Delta$ will be in $L^+_2(n)$ (respectively $L^-_2(n)$)
if $ \jacs{\Delta}{n} = 1$ (respectively $-1$).  Notice that if $ \jacs{\Delta}{n} = 0$, $f(x)$ is not a liar since it fails 
the Integer Divisibility step.  
Let ${\rm Frob}_2(y, f(x))$ be the set of degree-$2$ Frobenius pseudoprimes with respect to $f(x)$, up to bound $y$, 
and similarly divide them into $+$ and $-$ sets according to the Jacobi symbol.  Further, let ${\rm Frob}_2(f(x))$
be the (possibly infinite) set of all such pseudoprimes.
As an abuse of notation, the same symbols will be used for the size of each set.

The main goal of this work is to generalize \cite[Theorem 2.1]{ErdosPomerance01}, which bounds the average 
number of Fermat liars.  We prove the following two theorems.

\begin{theorem}\label{thm:th1}
For all $\alpha$ satisfying Proposition \ref{prop:lb1}, in particular $\alpha \leq \frac{10}{3}$, we have that as $y\rightarrow \infty$
$$
y^{3-\alpha^{-1}-o(1)} \leq \sum_{n \leq y} L_2^+(n) \leq y^3 \cdot \mathcal{L}(y)^{-1 + o(1)}
$$
where the sum is restricted to odd composite $n$.
Moreover, the same bounds hold if we replace $L_2^+(n)$ by $L_2(n)$.
Here $\mathcal{L}(y) = {\rm exp}( (\log{y})(\log\log\log{y})/\log\log{y})$, with $\log$ denoting the natural logarithm.
\end{theorem}

\begin{theorem}\label{thm:th2}
For all $\alpha$ satisfying Proposition \ref{prop:lb2}, in particular $\alpha \leq \frac{4}{3}$, we have that as $y\rightarrow \infty$
$$
y^{3-\alpha^{-1}-o(1)} \leq \sum_{n \leq y} L_2^-(n) \leq y^3 \cdot \mathcal{L}(y)^{-1 + o(1)}
$$
where the sum is restricted to odd composite $n$. 
\end{theorem}

As a comparison, if $n$ is prime then the size of $L_2(n)$ is $(n-1)^2$, 
$L_2^+(n) = \frac{1}{2}(n-1)(n-2)$,  and $L_2^-(n) = \frac{1}{2}n(n-1)$.  Thus the average count of liars 
for composites is rather large.

\begin{rmk}
We obtain the same results if we restrict to composite $n$ 
coprime to some fixed value.
\end{rmk}

These theorems count pairs $(f(x), n)$ where $n \leq y$ and $n$ is a degree-$2$ Frobenius pseudoprime 
with respect to $f(x)$.  We thus have the following corollary on the average count of degree-$2$ Frobenius pseudoprimes 
with Jacobi symbol $-1$.

\begin{corollary} \label{cor:psp_count}
Suppose $\alpha$ satisfies the conditions outlined in Theorem \ref{thm:th2}.  Then as $y\rightarrow \infty$ we have
$$
\sum_{a,b \leq y} {\rm Frob}_2^-(y, x^2+ax+b) \geq y^{3-\alpha^{-1}-o(1)} \enspace .
$$
\end{corollary}

In \cite[Section 8]{Grantham01}, Grantham offers \$6.20 for exhibiting a Frobenius pseudoprime with respect to $x^2+5x+5$
that is congruent to $2$ or $3$ modulo $5$.  The proper generalization for these Grantham challenge pseudoprimes 
are the sets ${\rm Frob}_2^-(x^2+ax+b)$, since the condition of being $2,3 \pmod{5}$ is equivalent to 
$\jacs{5}{n} = -1$.  
Grantham later proved \cite[Theorem 2.1]{Grantham10} that the sets ${\rm Frob}(f(x))$ are infinite for
all monic, squarefree polynomials $f(x) \in \Z[x]$, but his construction is limited to composite $n$ 
for which $(\Delta \mid n) = 1$ and $(\Delta \mid p) = 1$ for all $p \mid n$.  Our work is limited 
to degree $2$ polynomials, but expands the cases to include Jacobi symbol $-1$ both for $n$ and for $p \mid n$.
Corollary \ref{cor:psp_count} is consistent with the conjecture that the sets ${\rm Frob}_2^-(f(x))$ are 
infinite as well, and provides good evidence that there are infinitely many Grantham challenge pseudoprimes.

Further motivation for the present work comes from other challenge pseudoprimes.
PSW challenge pseudoprimes \cite[Section A12]{Guy04}, also known as \$620 problem numbers, 
 are composite $n$ that are simultaneously 
base-$2$ Fermat pseudoprimes, Fibonacci pseudoprimes, and congruent to $2,3$ modulo $5$.
Potentially even more rare are Baillie pseudoprimes \cite[Section 6]{BaillieWagstaff01}
(also called Baillie-PSW pseudoprimes due to the challenge posed in \cite[Section 10]{PSW01}), 
composite $n$ 
that are simultaneously base-$2$ strong pseudoprimes and strong Lucas pseudoprimes with respect to a polynomial 
$x^2-Px+Q$ chosen in a prescribed way to ensure $\jacs{P^2-4Q}{n} = -1$.
Though it is unresolved whether these sought-after numbers 
are Frobenius pseudoprimes, strong Frobenius pseudoprimes, or
something more restrictive still, quadratic Frobenius pseudoprimes provide a natural generalization 
for the types of conditions requested.

From this we conclude that the division of $L_2(n)$ into $\jacs{\Delta}{n} = \pm 1$ cases is of fundamental importance, 
and in particular that bounding $\sum_{n \leq y} L_2^-(n)$ is of strong interest.

Since ${\rm Frob}_2(x^2-Px+Q)$ is a subset of the set of $(P,Q)$-Lucas pseudoprimes 
\cite[Theorem 4.9]{Grantham01}, Corollary \ref{cor:psp_count} gives an immediate lower bound
 on the average count 
of Lucas pseudoprimes.  We do not explore the connection to Lucas pseudoprimes further in this work.

\section{Degree-$2$ Frobenius pseudoprimes}

This work focuses on the degree $2$ case.  We reproduce the definition and give some basic facts 
about Frobenius pseudoprimes and liars.  From now on $n$ will be an odd composite natural number.

\begin{definition} \label{def:2Frob}
Let $f(x) \in \Z[x]$ be a degree $2$ monic polynomial with discriminant $\Delta$, and let $n$ be an odd composite.
Then $n$ is a degree-$2$ Frobenius pseudoprime with respect to $f(x)$ if the following four conditions hold.  
\begin{enumerate}
\item (Integer Divisibility) We have $\gcd(n, f(0)\Delta) = 1$.

\item (Factorization) 
Let $F_1(x) = \gcmd(x^n-x, f(x))$, $f_1(x) = f(x)/F_1(x)$, $F_2(x) = \gcmd(x^{n^2}-x, f_1(x))$, 
and $f_2(x) = f_1(x)/F_2(x)$.  All these polynomials exist and $f_2(x) = 1$.

\item (Frobenius) We have $F_2(x) \mid F_2(x^n)$.

\item (Jacobi) We have $(-1)^S =  \jacs{\Delta}{n}$, where $S = \deg(F_2(x))/2$.
\end{enumerate}
Alternatively, in this case we call $f(x)$ 
a degree-$2$ Frobenius liar with respect to $n$.
\end{definition}

The first condition ensures that $\Delta \neq 0$ and $0$ is not a root of $f(x)$.  Since the discriminant is 
nonzero, $f(x)$ is squarefree.  Thus the roots of $f(x)$ are nonzero and distinct modulo $p$ for all $p \mid n$.

\begin{example}
Consider $f(x) = x^2-1$ with $\Delta = 4$.  If $n$ is odd, $F_1(x) = f(x)$ and $F_2(x) = 1$, 
so the Frobenius step is trivially satisfied.  Since $S = 0$, $n$ will be a Frobenius pseudoprime as long as 
$ \jacs{\Delta}{n} = 1$.  Since $4$ is a square modulo $n$ for all $n \geq 5$, we conclude that all odd $n \geq 5$ 
have at least one degree-$2$ Frobenius liar.
\end{example}

\begin{example}
Next consider $f(x) = x^2+1$ with $\Delta = -4$.  Observe that $n = 1 \pmod{4}$ if and only if 
$(-1)^{(n-1)/2} = 1$, which is true if and only if $\gcmd(x^n-x, f(x)) \neq 1$.  In this case $F_2(x) = 1$
and $(-1)^S = 1 = \jacs{-1}{n} =  \jacs{\Delta}{n}$.  In the other case, $n = 3 \pmod{4}$ if and only if 
$\gcmd(x^n-x, f(x)) = 1$.  However, $(-1)^{(n^2-1)/2} = 1$ and so $\gcmd(x^{n^2}-x, f(x)) = f(x)$.
For the Frobenius step, we know $x^2+1 \mid x^{2n}+1$ since if $a$ is a root of $x^2+1$, $n$ odd implies 
that $(a^2)^n = -1$ and hence $a$ is also a root of $x^{2n}+1$.  Finally, the Jacobi step is satisfied since 
$(-1)^S = -1 =  \jacs{\Delta}{n}$.  
This demonstrates that $x^2+1$ is also a liar for all odd 
composite $n$.  The minimum number of degree-$2$ Frobenius liars for odd composite $n$
is in fact $2$, first achieved by $n=15$.
\end{example}


If we fix $n$ and instead restrict to liars with $ \jacs{\Delta}{n} = -1$ then it is possible that no such liars exist.
See Section \ref{subsec:vanish} for a more in-depth discussion of this case.

We next give several reinterpretations of the conditions under which a number $n=\prod_i p_i^{r_i}$ is a degree-$2$ Frobenius pseudoprime with respect to a polynomial $f$.
We treat cases $\jac{\Delta}{n} = +1$ and $\jac{\Delta}{n} = -1$ separately.

\subsection{The case $\jac{\Delta}{n} = +1$}

Supposing we already know that $\jac{\Delta}{n} = +1$, $n$ is a degree-$2$ Frobenius pseudoprime with respect to $f(x)$ if and only if
\begin{enumerate}
\item (Integer Divisibility) we have $\gcd(n, f(0)\Delta) = 1$, and 
\item (Factorization) $\gcmd(x^n-x, f(x)) = f(x) \pmod{n}$.
\end{enumerate}
All other conditions follow immediately. In particular, because $f(x) \mid x^n-x$ modulo $n$, it is not possible for the Euclidean algorithm to discover any non-trivial factors of $n$. We observe that these conditions can be interpreted locally, giving us the following result.

\begin{proposition} \label{newdef_plus}
Positive integer $n = \prod_i p_i^{r_i}$ satisfies Definition \ref{def:2Frob} in the case $\jacs{\Delta}{n} = 1$ if and only if 
\begin{enumerate}
\item (Integer Divisibility) $\Delta$ is a unit modulo $n$ and $0$ is not a root of $f(x)$ modulo $p_i$ for all $i$, and 
\item (Factorization) $\gcmd(x^n-x, f(x)) = f(x) \pmod{p_i^{r_i}}$ for all $i$.
\end{enumerate}
\end{proposition}
\begin{proof}
First assume that $n$ is a degree-$2$ Frobenius pseudoprime with respect to $f(x)$ according to Definition \ref{def:2Frob}
and that $\jacs{\Delta}{n} = 1$.  Then $\gcd(n, f(0) \Delta) = 1$, so $\gcd(\Delta, n) = 1$ making $\Delta$ a unit, and 
$\gcd(f(0), n) = 1$.  It follows that $f(0) \neq 0 \pmod{p}$ for all $p \mid n$.

The Jacobi condition in Definition \ref{def:2Frob} along with the assumption that $\jacs{\Delta}{n} = 1$ ensures $S = 0$
and so ${\rm deg}(F_2(x)) = 0$.  All the polynomials in condition (2) are monic, so $F_2(x) = 1$, which implies 
$f_1(x) = 1$, so that $\gcmd(x^n - x, f(x)) = f(x)$.  Since this identity is true modulo $n$, it is true modulo $p_i^{r_i}$ for all $i$.

Conversely, if $\gcmd(x^n - x, f(x)) = f(x) \pmod{p_i^{r_i}}$ for all $i$, then the identity is true modulo $n$ by the Chinese remainder 
theorem.  It follows that $f_1(x) = 1$ and so $F_2(x) = 1$.  Thus condition (2) of Definition \ref{def:2Frob} is true, condition (3)
follows trivially, and condition (4) is true since $S = 0$.

We are assuming that $\Delta$ is a unit modulo $n$, from which it follows that $\gcd(\Delta, n) = 1$.
Furthermore, $f(0) \neq 0 \pmod{p}$ for all $p \mid n$ implies $\gcd(f(0), n) =1$.  Thus condition (1) is satisfied.
\end{proof}

\subsection{The Case $\jac{\Delta}{n} = -1$}

When $\jacs{\Delta}{n} = -1$ we need a couple more conditions.

\begin{proposition} \label{newdef_minus}
Positive integer $n = \prod_i p_i^{r_i}$ satisfies Definition \ref{def:2Frob} in the case $\jacs{\Delta}{n} = -1$ if and only if it satisfies the following conditions:
\begin{enumerate}
\item (Integer Divisibility) discriminant $\Delta$ is a unit modulo $n$ and $0$ is not a root of $f(x)\pmod{p_i}$ for all $i$,
\item (Factorization 1) $\gcmd(x^n-x, f(x)) = 1 \pmod{p_i^{r_i}}$ for all $i$,
\item (Factorization 2) $\gcmd(x^{n^2}-x, f(x)) = f(x) \pmod{p_i^{r_i}}$  for all $i$,
\item (Frobenius) if $\alpha$ is a root of $f(x)$ modulo $p_i^{r_i}$, then so too is $\alpha^n$ for all $i$.
\end{enumerate}
In particular, these conditions are sufficient to ensure that $\gcmd(x^n-x, f(x))$ and $\gcmd(x^{n^2}-x, f(x))$ exist modulo $n$.
\end{proposition}
\begin{proof}
Following the argument from Proposition \ref{newdef_plus}, condition (1) from Definition \ref{def:2Frob}
holds if and only if $\Delta$ is a unit modulo $n$ and $0$ is not a root of $f(x) \pmod{p_i}$ for all $i$.

Now, if we assume $n$ satisfies Definition \ref{def:2Frob}, then by condition (4) we must have $S = 1$ and hence 
$\deg(F_2(x)) = 2$.  Thus $\gcmd(x^{n^2}-x, f_1(x)) = f(x)$, and since $f_2(x) = 1$ we further have $f_1(x) = f(x)$.
This is only possible if $\gcmd(x^n - x, f(x)) = 1$.  Since these identities hold modulo $n$, they hold modulo $p_i^{r_i}$ 
for all $i$.  Finally, $F_2(x) \mid F_2(x^n)$ means $f(x) \mid f(x^n) \pmod{n}$ and hence that $\alpha^n$ 
is a root of $f(x)$ modulo $p_i^{r_i}$ whenever $\alpha$ is.

Conversely, assume $n$ satisfies conditions (2), (3), (4) from the statement of the proposition.  
By the Chinese remainder theorem, conditions (2) and (3) mean that $\gcmd(x^n-x, f(x)) = 1 \pmod{n}$
and $\gcmd(x^{n^2}-x, f(x)) = f(x) \pmod{n}$.  In the language of Definition \ref{def:2Frob}, we have $F_1(x) = 1$, 
$F_2(x) = f(x)$, and $f_2(x) = 1$ as required.  It follows that the Jacobi step is satisfied.  And finally,
condition (3) means that $f(x) \mid f(x^n) \pmod{p_i^{r_i}}$ for all $i$, and so the Frobenius step is satisfied modulo $n$.

If all $\gcmd$ calculations exist modulo $n$, then they exist modulo $p_i^{r_i}$ for all $i$, 
so to finish the proof we need to show that the latter condition is sufficient to ensure 
$\gcmd(x^n-x, f(x))$ and $\gcmd(x^{n^2}-x, f(x))$ exist.  Since $\gcmd(x^n-x, f(x)) = 1 \pmod{p_i^{r_i}}$
for all $i$, by \cite[Proposition 3.4]{Grantham01} we know that $\gcmd(x^n-x, f(x)) = 1 \pmod{n}$
and thus exists.  If $\gcmd(x^{n^2}-x, f(x)) = f(x) \pmod{p_i^{r_i}}$, then $p_i^{r_i}$ divides 
$x^{n^2}-x - f(x)g_i(x)$ for some polynomial $g_i(x)$.  However, using the Chinese remainder theorem 
on each coefficient in turn, we can construct  a polynomial $g(x) \in (\Z/n\Z)[x]$ such that 
$g(x) = g_i(x) \pmod{p_i^{r_i}}$ for all $i$.  Then for all $i$, $p_i^{r_i}$ divides 
$x^{n^2}-x - f(x)g(x)$ and hence $n$ divides $x^{n^2}-x-f(x)g(x)$.  This shows that 
$\gcmd(x^{n^2}-x, f(x)) = f(x) \pmod{n}$, and in particular that it exists.
\end{proof}

\begin{rmk}
Earlier we commented that if the $\gcmd$ does not exist, the Euclidean algorithm will find a factor of $n$.
Here we note that even if the $\gcmd$ exists, the Euclidean algorithm might detect a factor of $n$ 
while computing it.  For example, consider $n = 14$ and $\gcmd(x^{14}-x, x^2+x+1)$ modulo $14$.
When dividing $x^{14}-x$ by $x^2+x+1$ we get a remainder of $-2x - 1$, so the second step of the Euclidean 
algorithm will attempt to invert $2$ and find $2$ as a factor of $14$.  However, 
$4(x^2+x+1) + (2x+1)(-2x - 1) = 3$, a unit modulo $14$, and thus $\gcmd(x^{14}-x, x^2+x+1) = 1$.

That said, for the calculations involved in checking for degree-$2$ Frobenius pseudoprimes this failure mode 
can only happen in the  $\jac{\Delta}{n} = -1$ case and only if either $n$ is even or if one of the conditions (1-4) of Proposition \ref{newdef_minus} would already fail. When $n$ is even, it will only discover a power of $2$ (and the complementary factor).  The rest of this remark justifies these claims.

First, assume the Euclidean algorithm would discover factors of $n$.  If the Factorization $1$ and Factorization $2$ conditions are passed then it implies there exist primes $p_i$ and $p_j$, such that at some iteration of the Euclidean algorithm to compute $\gcmd(x^n-x, f(x))$, the degrees of the polynomials being considered differ.

We note that given $\gcmd(x^n-x, f(x)) = 1 \pmod{n}$ we must have for each $p \mid n$ that
\[ x^n-x = f(x)g(x) + ax+b \pmod{p} \]
where either $a=0$ and $b$ is a unit, or $a$ is a unit.
However, if $a=0$, then condition (Frobenius) implies the roots of $f(x)$ modulo $p$ are $\alpha$ and $\alpha+b$. But this can only happen for $p=2$.
In particular, if $n$ is odd, then we must have that $a\neq 0$ is a unit for all $p \mid n$ and thus
\[ x^n-x = f(x)g(x) + ax+b \pmod{n} \enspace . \]
Given that $\gcmd(x^n-x, f(x)) = 1 \pmod{n}$ we then have that
\[ f(x) = (ax+b)h(x) + e \pmod{n} \]
where $e$ is a unit.
It follows that the only possible discrepancy between $p_i$ and $p_j$ is if one of the primes is $2$.

Finally, the Euclidean algorithm will not discover a factor of $n$ while computing $\gcmd(f(x),g(x))$ if the result is $f(x)$. 
\end{rmk}

\section{Monier formula for degree-$2$ Frobenius pseudoprimes}

In this section we give explicit formulas, analogous to those of Monier \cite[Proposition 1]{Monier01}, 
for the quantity $L_2(n)$ of polynomials $f(x)$ modulo $n=\prod_i p_i^{r_i}$ for which $n$ is a 
degree-$2$ Frobenius pseudoprime.
The key step will be reinterpreting the conditions of the previous section in terms of conditions on the 
roots $\alpha$ and $\beta$ of $f(x)$ modulo $p_i^{r_i}$ for each $i$.

As in the previous section, it shall be useful to distinguish the cases $\jac{\Delta}{n} = \pm1$, and as such we will give separate formulas for $L_2^\pm(n)$.

\begin{nota}
For each fixed value of $n$, denote by $L_2^+(n)$ the total number of quadratic polynomials $f \pmod{n}$ such that $(f,n)$ is a liar pair and $\jac{\Delta}{n} = +1$.

For each fixed value of $n$, denote by $L_2^-(n)$ the total number of quadratic polynomials $f \pmod{n}$ such that $(f,n)$ is a liar pair and $\jac{\Delta}{n} = -1$.
\end{nota}

At the heart of the formula is the size and structure of the ring $R:= (\Z/p^r\Z)[x]/\langle f(x) \rangle$, so we spend 
a little time discussing some basic facts.

Recall that in the case where $r=1$, if $\jac{\Delta}{p} = 1$ then $R \simeq \F_p \times \F_p$ and $|R^{\times}| = (p-1)^2$, 
while if $\jac{\Delta}{p} = -1$ then $R \simeq \F_{p^2}$ and $R^{\times}$ is cyclic of order $p^2-1$.  
When $r > 1$ we have the canonical surjective homomorphism 
$$
\phi:  (\Z/p^r\Z)[x]/\langle f(x) \rangle \rightarrow (\Z/p\Z)[x]/\langle f(x) \rangle
$$
and a similar map on the unit groups.  Furthermore, $f(x)$ will split in $\Z/p^r\Z$ if and only if $\jac{\Delta}{p}=1$.
Thus $|R^{\times}| = p^{2r-2}(p-1)^2$ if $\jac{\Delta}{p} = 1$ and $|R^{\times}| = p^{2r-2}(p^2-1)$ if $\jac{\Delta}{p} = -1$.  In the latter case, since $R^{\times}$ maps surjectively onto a cyclic group 
of order $p^2-1$ with kernel a $p$-group, it has a cyclic subgroup $C$ of order $p^2-1$.  This fact follows 
from the fundamental theorem of abelian groups \cite[Exercise 1.43]{Lang02}, and implies that 
there is a section of $\phi$ yielding a bijective homomorphism from $C$ to $(\Z/p\Z)[x]/\langle f(x) \rangle$.

\subsection{The case $\jac{\Delta}{n} = +1$}

We note that in this case there must be an even number of primes $p_i$ for which $r_i$ is odd and $\jac{\Delta}{p} = -1$.

In order to count the number of $f(x)$ modulo $n$, we shall count for each $i$ the number of modulo $p_i^{r_i}$ liars for which 
$\jac{\Delta}{p} = \pm1$. By the Chinese remainder theorem, the desired count is then the product for all combinations which ensure the above parity condition.  

\begin{lemma} \label{lem:L2++}
Suppose $p^r||n$. The number of degree $2$ polynomials over  $(\Z/p^{r}\Z)$ with $\jac{\Delta}{n} = +1$ and $\jac{\Delta}{p} = +1$ 
for which $n$ is a quadratic Frobenius pseudoprime at $p$ is exactly
\[ L_2^{++}(n,p) = \frac{1}{2}\left(\gcd(n-1,p-1)^2 - \gcd(n-1,p-1)\right) \enspace . \]
\end{lemma}
\begin{proof}
Referring to Proposition \ref{newdef_plus}, $\gcmd(x^n - x, f(x)) = f(x) \pmod{p^r}$ means that $\alpha^n = \alpha$
and $\beta^n = \beta$ modulo $p^r$ for roots $\alpha, \beta$ of $f(x)$.  In addition, the roots are distinct and nonzero
by the integer divisibility condition.

The group $(\Z/p^r\Z)^{\times}$ is cyclic, so it has $\gcd(n-1, p^{r-1}(p-1)) = \gcd(n-1,p-1)$
 elements whose order divides both $n-1$ and $p-1$.
Choosing two such elements, which are not congruent modulo $p$, gives the result.
\end{proof}

\begin{lemma}
Suppose $p^r||n$. The number of degree $2$ polynomials over  $(\Z/p^{r}\Z)$ with $\jac{\Delta}{n} = +1$ and $\jac{\Delta}{p} = -1$ 
for which $n$ is a quadratic Frobenius pseudoprime at $p$ is exactly
\[ L_2^{+-}(n,p) = \frac{1}{2}\left(\gcd(n-1,p^2-1) - \gcd(n-1,p-1)\right) \enspace . \]
\end{lemma}
\begin{proof}
We again refer to Proposition \ref{newdef_plus}.  Since $\jac{\Delta}{p} = -1$, $R:= (\Z/p^r\Z)[x]/\langle f(x) \rangle$
maps surjectively onto $\F_{p^2}$ and the cofactor has size $p^{2r-2}$.  Furthermore, the distinct, nonzero roots 
$\alpha, \beta$ of $f(x)$ are not lifts of elements of $\F_p$, and $\alpha^{p^r} = \beta \pmod{p^r}$.
The factorization condition implies that $\alpha^n = \alpha \pmod{p^r}$,
so that the order of $\alpha$ in $R^{\times}$ divides $n-1$.

All elements of $R^{\times}$ have order dividing $p^{2r-2}(p^2-1)$.
Hence the number of options for $\alpha$ is exactly $\gcd(p^2-1,n-1)-\gcd(p-1,n-1)$, and we divide by $2$ since 
the polynomial $f(x)\pmod{p^r}$ is symmetric in $\alpha$ and $\beta$.
\end{proof}

In order to capture the requirement that we have an even number of contributions from primes where $\jac{\Delta}{p} = -1$ with $r_i$ odd, we anti-symmetrize with respect to these terms to obtain the formula for $L_2^+(n)$.  

\begin{theorem} \label{thm:L2+}
The number of degree $2$ polynomials over  $(\Z/n\Z)$ with $\jac{\Delta}{n} = +1$ 
for which $n$ is a quadratic Frobenius pseudoprime is exactly
\begin{align*}
\frac{1}{2}\prod_{i}\left( L_2^{++}(n,p_i)+ L_2^{+-}(n,p_i) \right)
   +\frac{1}{2}\prod_{2\mid r_i}\left( L_2^{++}(n,p_i)+ L_2^{+-}(n,p_i) \right)  \prod_{2 \nmid r_i} \left( L_2^{++}(n,p_i)- L_2^{+-}(n,p_i) \right)
   \enspace .
\end{align*}
\end{theorem}

\begin{corollary} \label{plusformula}
If $n$ is squarefree, the formula in Theorem \ref{thm:L2+} becomes
\begin{align*}
L_2^+(n)= \frac{1}{2} \prod_{p \mid n}& \frac{1}{2}\left( \gcd(n-1,p^2-1) + \gcd(n-1,p-1)^2  - 2\gcd(n-1,p-1)) \right)\\
  & +  \frac{1}{2}\prod_{p \mid n} \frac{1}{2}\left( \gcd(n-1,p-1)^2 -  \gcd(n-1,p^2-1) ) \right) \enspace .
\end{align*}
\end{corollary}

\subsection{The case $\jac{\Delta}{n} = -1$}

In this case there must be an odd number of primes $p_i$ for which $r_i$ is odd and 
$\jac{\Delta}{p} = -1$.  As above, the liar count is first computed separately for each $p$.

\begin{lemma}
The number of degree $2$ polynomials over  $(\Z/p^{r}\Z)$ with $\jac{\Delta}{n} = -1$ and $\jac{\Delta}{p} = +1$ 
for which $n$ is a quadratic Frobenius pseudoprime at $p$ is exactly
\[  L_2^{-+}(n,p) = \frac{1}{2}\left(\gcd(n^2-1,p-1) - \gcd(n-1,p-1)\right) \enspace . \]
\end{lemma}
\begin{proof}
Since $\jac{\Delta}{p} = 1$, the roots $\alpha, \beta$ of $f(x)$ are in $(\Z/p^r\Z)$.
Referring to Proposition \ref{newdef_minus}, the roots are distinct and nonzero by the integer divisibility 
condition.  Furthermore, $\gcmd(x^n - x, f(x)) = 1$ means that $\alpha^n \neq \alpha \pmod{p^r}$, 
but we do have $\alpha^{n^2} = \alpha \pmod{p^r}$ by the factorization 2 condition.  The Frobenius 
condition implies $\alpha^n$ is a root of $f(x)$, and thus $\alpha^n = \beta$.

The group $(\Z/p^r\Z)^{\times}$ is cyclic of order $p^{r-1}(p-1)$, and so the number of elements 
with order dividing both $n^2-1$ and $p^{r-1}(p-1)$ is $\gcd(n^2-1, p-1)$.  We subtract off the subset 
of elements with order dividing $n-1$, then divide by $2$ since $f(x) \pmod {p^r}$ is symmetric 
in $\alpha$ and $\beta$.
\end{proof}

\begin{lemma}\label{Lem:L2--}
The number of degree $2$ polynomials over  $(\Z/p^{r}\Z)$ with $\jac{\Delta}{n} = -1$ and $\jac{\Delta}{p} = -1$ 
for which $n$ is a quadratic Frobenius pseudoprime at $p$ is exactly
\[  L_2^{--}(n,p) = \frac{1}{2}\left(\gcd(p^2-1,n^2-1,n-p)- \gcd(n-1,p-1)\right) \enspace . \]
\end{lemma}
\begin{proof}
Since $\jac{\Delta}{p} = -1$, $R:= (\Z/p^r\Z)[x]/\langle f(x) \rangle$ maps surjectively onto $\F_{p^2}$
and $R^{\times}$ has order $p^{2r-2}(p^2-1)$.  Furthermore, roots $\alpha, \beta$ of $f(x)$ are not in $\Z/p^r\Z$, and by the divisibility condition  in Proposition \ref{newdef_minus} we know 
those roots are distinct units modulo $p$.  
The factorization conditions tell us that $\alpha^{n^2} = \alpha \pmod{p^r}$
and the Frobenius condition implies $\alpha^n = \beta \pmod{p^r}$.

We claim a further relation on the roots, namely that $\alpha^p = \beta \pmod{p}$
implies $\alpha^p = \beta \pmod{p^r}$.  If $\alpha^p = \beta \pmod{p}$, then 
$\alpha^p/\beta \in {\rm Ker}(\phi)$, a $p$-group, and thus has order divisible by $p$ unless it is the identity 
in $R^{\times}$.
However, the multiplicative orders of $\alpha, \beta$ divide $n^2-1$, and thus $\alpha^p/\beta \in C$, 
the cyclic group in $R^{\times}$ of order $p^2-1$.  It follows that 
the order of $\alpha^p/\beta$ in $R^{\times}$
must be $1$, making it the identity modulo $p^r$.

We conclude that the order of $\alpha$ in $R^{\times}$ must divide $n^2-1$, $p^{2r-2}(p^2-1)$, and $n-p$.
The number of options for $\alpha$ is thus exactly $\gcd(p^2-1,n^2-1,n-p)-\gcd(p-1,n-1)$, and we 
divide by $2$ since the polynomial $f(x)\pmod{p^r}$ is symmetric in $\alpha$ and $\beta$.
\end{proof}

In order to capture the requirement that we have an odd number of contributions from primes where $\jac{\Delta}{p} = -1$ when $r_i$ is odd, we anti-symmetrize with respect to these terms to obtain the formula for $L_2^-(n)$.  

\begin{theorem} \label{thm:negliars}
The number of degree $2$ polynomials over  $(\Z/n\Z)$ with $\jac{\Delta}{n} = -1$ 
for which $n$ is a quadratic Frobenius pseudoprime is exactly
\begin{align*}
\frac{1}{2}\prod_{i}\left( L_2^{-+}(n,p_i)+ L_2^{--}(n,p_i) \right)
   -\frac{1}{2}\prod_{2\mid r_i}\left( L_2^{-+}(n,p_i)+ L_2^{--}(n,p_i) \right)  \prod_{2\nmid r_i}
  \left( L_2^{-+}(n,p_i) -L_2^{--}(n,p_i) \right) \enspace .
\end{align*}
\end{theorem}

\begin{corollary} \label{minusformula}
If $n$ is squarefree, the formula in Theorem \ref{thm:negliars} becomes
\begin{align*}
 L_2^-(n) = \frac{1}{2} \prod_{p \mid n}&\frac{1}{2}\left(  \gcd(n^2-1,p-1)+ \gcd(n^2-1,p^2-1,n-p) -2\gcd(n-1,p-1) \right)\\& -\frac{1}{2}  \prod_{p \mid n} \frac{1}{2}\left(  \gcd(n^2-1,p-1)- \gcd(n^2-1,p^2-1,n-p)  \right) \enspace .
\end{align*}
\end{corollary}

\subsection{Upper bounds}

In this section we give simpler upper bounds for $L_2^+(n)$ and $L_2^-(n)$, 
which will be needed in Section \ref{sec:upper}.

\begin{lemma}\label{lemma:liarcountupperbounds}
If $n$ is a composite integer then
\begin{align*}
& L_2^+(n) \leq \prod_{p \mid n} \max( \gcd(n-1, p^2-1), \gcd(n-1, p-1)^2) \mbox{ and } \\
& L_2^-(n) \leq \prod_{p \mid n} \gcd(n^2-1, p^2-1) \enspace .
\end{align*}
\end{lemma}
\begin{proof}
For each prime factor $p$ of $n$, we choose the greater of $L_2^{++}(n,p)$
and $L_2^{+-}(n,p)$.  That is,
$$
L_2^+(n) \leq \prod_i 2 \cdot \max(L_2^{++}(n, p_i), L_2^{+-}(n,p_i)) 
 \leq  \prod_{p \mid n} \max( \gcd(n-1,p-1)^2, \gcd(n-1,p^2-1)) \enspace .
$$

For $L_2^-(n)$ a similar argument gives the simpler upper bound  
$$
\prod_{p \mid n} \max( \gcd(n^2-1, p^2-1, n-p), \gcd(n^2-1, p-1)) 
\leq \prod_{p \mid n} \gcd(n^2-1, p^2-1) \enspace .
$$
\end{proof}

\subsection{The vanishing of $L_2^-(n)$}\label{subsec:vanish}

A major theme of this work is that odd composites have many quadratic Frobenius 
liars on average, even if we restrict to the case $\jac{\Delta}{n} = -1$.  With this 
in mind, it is useful to note that $L_2^{-}(n)$ can be $0$.  For example, $L_2^-(9) = 0$
and $L_2^-(21) = 0$.

As a first general example, write $n=ps$ with $\gcd(p,s) = 1$.  If the quantities
\[\gcd(p^2-1,n^2-1,n-p)- \gcd(n-1,p-1)  \quad \text{and} \quad \gcd(n^2-1,p-1) - \gcd(n-1,p-1) \]
are both zero, then it is immediate from Theorem \ref{thm:negliars} that $L_2^-(n) = 0$.
These conditions are met if whenever $\ell^r \mid \gcd(p^2-1,n^2-1,n-p)$ or $\ell^r \mid \gcd(n^2-1,p-1)$ we also have $\ell^r \mid \gcd(n-1,p-1)$.
For odd primes $\ell \mid p^2-1$ this is accomplished by the requirement that
\[ s \neq -p^{-1} \pmod \ell \enspace , \] 
as this implies that if $\ell\mid p^2-1$ then $\ell \nmid sp+1$.
For the prime $2$, if we write $p = 1+2^r \pmod{2^{r+1}}$ then the requirement
\[ s = 1+2^r \pmod{2^{r+1}} \]
implies the exact power of $2$ dividing each of 
$\gcd(p^2-1,n^2-1,n-p)$, $\gcd(n^2-1,p-1)$, and $\gcd(n-1,p-1)$ is $2^r$.

A more general example comes from Carmichael numbers, which are squarefree $n$ 
with $\gcd(n-1, p-1) = p-1$ for all primes $p \mid n$.

\begin{rmk}
If $n$ is a classical Carmichael number, then $L_2^{-+}(n,p) =0$ for all $p$ and 
\[ L_2^-(n) =  \prod_{p \mid n} \frac{1}{2} \left( \gcd(n^2-1,p^2-1,n-p) -\gcd(n-1,p-1) \right) \]
if $n$ has an odd number of prime factors, and $0$ otherwise (see Corollary \ref{minusformula}).
In particular, the only $f$ for which $(f,n)$ is a liar pair with $\jac{\Delta}{n} = -1$ have $f$ inert at all primes dividing $n$.
Furthermore, if $n=1\pmod{4}$ then for each $p \mid n$ with $p=3\pmod{4}$ we naively estimate the probability that $L_2^{--}(n,p) = 0$ as $\prod_{\ell \mid p+1}' \frac{\ell-2}{\ell-1}$, where the product is over odd primes $\ell$.
\end{rmk}

As a final example, let $n$ be a rigid Carmichael number of order $2$ in the sense of \cite{Howe00}, so that $n$ is squarefree and 
$p^2-1 \mid n-1$ for every prime factor $p$ of $n$.  Then $\gcd(n^2-1, p^2-1, n-p) = \gcd(n-1, p-1)$
and $\gcd(n^2-1, p-1) = \gcd(n-1, p-1)$, so that $L_2^-(n) = 0$.

\section{Number theoretic background}

\begin{nota}
Let $L$ be an upper bound for Linnik's constant.  That is, the constant $L$ satisfies:
\[ \text{if } (a,m) = 1 \text{ then there exists } p= a \pmod{m} \text{ with } p < m^L \enspace . \]
It is known that $L\leq 5.$ (See \cite[Theorem 2.1]{Xylouris01})

For each value $x$ denote by
$M(x)$  the least common multiple of all integers up to $\frac{\log{x}}{\log\log{x}}$.

For each value $x$ and for each $\alpha > 0$ denote by $P_{\alpha}^{(+)}(x)$ the set
 \[   \left\{ \text{prime} ~ p < (\log{x})^{\alpha} ~ \text{such that} ~(p-1) \mid M(x) \right\}  \]
and by $P_{\alpha}^{(-)}(x)$ the set
 \[  \left \{ \text{prime} ~ p < (\log{x})^{\alpha} ~ \text{such that} ~(p^2-1) \mid M(x)  \right\} \enspace .  \]
Now, given functions $M_1(x)$ and $M_2(x)$ of $x$ which satisfy
\[ M(x) = M_1(x)M_2(x) \qquad \text{and} \qquad \gcd(M_1(x),M_2(x)) = 2 \]
we define for each value $x$ and for each $\alpha > 0$ the set
 \[ P_{\alpha}\left(M_1(x),M_2(x),x\right)  = \left\{ \text{prime} ~ p < (\log{x})^{\alpha} ~ \text{such that} ~(p-1) \mid M_1(x) \text{ and } (p+1) \mid M_2(x) \right\} \enspace .  \]
\end{nota}

\begin{proposition} \label{prop:Msize}
We have $M(x) = x^{o(1)}$ as $x\rightarrow \infty$.
\end{proposition}
\begin{proof}
We can estimate $M(x)$ by:
\[ \prod_{p< \frac{\log{x}}{\log\log{x}}}  p^{\floor{\frac{\log\log{x}-\log\log\log{x}}{\log p}}} <  \prod_{p< \frac{\log{x}}{\log\log{x}}} \frac{\log{x}}{\log\log{x}} = \left( \frac{\log{x}}{\log\log{x}} \right)^{\pi\left(\frac{\log{x}}{\log\log{x}}\right)} = x^{o(1)} \enspace . \]
\end{proof}

The next two propositions follow from results on the smoothness of shifted primes.  The conclusion is that the sets 
$P_{\alpha}^{(+)}(x)$ and $P_{\alpha}^{(-)}(x)$ are relatively large.  As a comparison, by the prime number theorem the asymptotic count 
of all primes $p < (\log{x})^{\alpha}$ is $\frac{(\log{x})^{\alpha}}{\alpha \log\log{x}}$.

\begin{proposition}\label{prop:lb1}
For all $\alpha \leq 10/3$ we have that $\abs{P_{\alpha}^{(+)}(x) } \geq (\log{x})^{\alpha-o(1)}$ as $x\rightarrow \infty$.

\end{proposition}
\begin{proof}
This is Theorem 1 of \cite{BakerHarman98} under the assumption that $1/\alpha > 0.2961$.
\end{proof}

\begin{proposition}\label{prop:lb2}
For all $\alpha \leq 4/3$ we have that $\abs{P_{\alpha}^{(-)}(x) } \geq (\log{x})^{\alpha-o(1)}$ as $x\rightarrow \infty$.
\end{proposition}
\begin{proof}
We apply \cite[Theorem 1.2]{DMT01}. The constant $\alpha=4/3$ arises from the formula $(g-1/(2k))^{-1}$ 
because $x^2-1$ has $k = 2$ factors of degree $g=1$, and $1$ is the highest degree among the irreducible 
factors.
\end{proof}

The next proposition is a novel contribution to the theory of pseudoprime construction.
Recall that $\omega(n)$ is the count of distinct prime factors of $n$.

\begin{proposition}\label{prop:lb2X}
Given $\alpha$ such that $\abs{P_{\alpha}^{(-)}(x) } \geq (\log{x})^{\alpha-o(1)}$ as $x\rightarrow \infty$, then there exist $M_1(x)$, $M_2(x)$ such that as $x\rightarrow \infty$ we have
\[  \abs{P_{\alpha}\left(M_1(x),M_2(x),x\right) } \geq (\log{x})^{\alpha-o(1)} \enspace . \]
\end{proposition}
\begin{proof}
Let $M$ be the fixed choice of $M(x)$ that follows from a fixed choice of $x$.
Each prime $p\in P_{\alpha}^{(-)}(x)$ is also in $P_{\alpha}\left((p-1)d_1,(p+1)d_2,x\right)$ for all pairs $(d_1, d_2)$  satisfying
\[ d_1d_2 = \frac{M}{p^2-1} \quad \text{and} \quad \gcd(d_1,d_2) = 1 \enspace . \]

The number of pairs $(M_1, M_2)$ satisfying the conditions laid out in the notation comment at the beginning 
of the section is
\[ 2^{\pi(\log{x}/\log\log{x})}   \]
since each prime up to $\frac{\log{x}}{\log\log{x}}$ is assigned to either $M_1$ or $M_2$.  
To count the number of choices for $d_1$ and $d_2$ we subtract from the exponent the count of prime factors of $p^2-1$.
This work yields 
\[\sum_{M_1,M_2} \abs{ P_{\alpha}\left(M_1,M_2,x\right) } 
= \sum_{p\in P_{\alpha}^{(-)}(x) } 2^{ \omega\left( \frac{M}{p^2-1} \right) }
> 2^{\pi \left( \frac{\log x}{\log\log x} \right) - \omega_{{\rm max}}(p^2-1)}(\log x)^{\alpha-o(1)} \]
where $\omega_{{\rm max}}(p^2-1)$ denotes the maximum number of distinct prime factors of $p^2-1$ for all $p$ under consideration.

For integers up to $x$, the integer $n$ that maximizes $\omega(n)$ is formed by taking the product of 
all small distinct primes.
 By the argument in \cite[Section 22.10]{HardyWright}, it follows that 
for $n \leq x$ we have $\omega(n) \leq (1 + o(1)) \frac{\log x}{\log\log{x}}$.  Thus

\[ 2^{\omega_{{\rm max}}(p^2-1)}  \leq  
2^{ \frac{ (1 + o(1))\alpha \log\log x}{\log\log( (\log x)^{\alpha})}}
 \leq (\log x)^{o(1)} \enspace . \]

Now, if $\abs{ P_{\alpha}(M_1, M_2, x)} \leq (\log x)^{\alpha - o(1)}$ for all pairs $(M_1, M_2)$ we would conclude 
that 
$$
\sum_{M_1, M_2} \abs{ P_{\alpha}(M_1, M_2, x)} \leq 2^{\pi\left( \frac{\log x}{\log\log x } \right)}(\log{x})^{\alpha - o(1)} \enspace ,
$$
but since this contradicts the earlier lower bound we instead conclude that $\abs{ P_{\alpha}(M_1, M_2, x)} \geq (\log x)^{\alpha - o(1)}$
for at least one pair $(M_1, M_2)$.
\end{proof}

\begin{rmk}
From the proof we expect the result will in fact hold for most choices of $M_1$ and $M_2$.

The proof we have given does not actually imply any relationship between $M_i(x)$ for different values of $x$. In particular, though one perhaps expects that that there exists a complete partitioning of all primes into two sets and that the $M_i$ are simply constructed by considering only those primes in the given range, we do not show this.
\end{rmk}

It is generally expected (see for example \cite[proof of Theorem 2.1]{ErdosPomerance01}) that the values $\alpha$ under consideration can be taken arbitrarily large. In particular we expect the following to hold.
\begin{conj}
In each of the above three propositions, the result holds for all $\alpha>0$.
\end{conj}

The following lemma will be useful in the next section.

\begin{lemma}\label{lem:cong-cases}
Fix $n$ and $p \mid n$. If $n=-1 \pmod{q}$ and $p=1 \pmod{q}$ for $q\ge 3$ then
\[ \gcd(n^2-1,p-1) - \gcd(n-1,p-1) > 0 \enspace . \]

If $n=-1 \pmod{q}$ and $p= -1 \pmod{q}$ for $q\ge 3$  then
\[ \gcd(n^2-1,p^2-1,n-p) - \gcd(n-1,p-1) > 0 \enspace . \]

If $n=p=1 \pmod{2}$ then
\[ \gcd(n-1,p-1)^2-\gcd(n-1,p-1) > 0 \enspace . \]
\end{lemma}
\begin{proof}
For  $n=-1 \pmod{q}$ and $p=1 \pmod{q}$ we have $q \mid  \gcd(n+1,p-1)$
while  $q \nmid \gcd(n-1,p-1)$.  

If $n=-1 \pmod{q}$ and $p= -1 \pmod{q}$ then $q \mid \gcd(n^2-1,p^2-1,n-p)$ and $q \nmid \gcd(n-1,p-1)$.

Finally, for $n=p=1 \pmod{2}$ it follows that $\gcd(n-1,p-1)> 1$, and so 
$$\gcd(n-1,p-1)(\gcd(n-1,p-1) - 1)$$ is nonzero.
\end{proof}


\section{Lower bounds on the average number of degree-$2$ Frobenius pseudoprimes}

In this section we will prove the lower bound portion of the two theorems in the introduction. Specifically we shall prove the following results.

\begin{theorem}\label{thm:one}
For any value of $\alpha > 1$ satisfying Proposition \ref{prop:lb1} we have the asymptotic inequality
\[ \sum_{n<x} L_2^+(n) \geq x^{3-\alpha^{-1} - o(1)} \]
as $x\rightarrow \infty$.
\end{theorem}

\begin{theorem}\label{thm:two}
For any value of $\alpha > 1$ satisfying Proposition \ref{prop:lb2} we have the asymptotic inequality
\[ \sum_{n<x} L_2^-(n) \geq x^{3-\alpha^{-1} - o(1)} \]
as $x\rightarrow \infty$.
\end{theorem}
The proofs of the above two theorems are at the end of this section. We shall first introduce some notation and prove several necessary propositions.

\begin{nota}
For fixed $0 < \epsilon < \alpha-1$ and for all $x > 0$ let
\begin{itemize}
\item $k^{(+)}_\alpha(x) =\floor{ \frac{\log{x} - L\log M}{\alpha\log\log{x}}}$ \enspace ,

\item $k^{(-)}_\alpha(x) =\floor{ \frac{\log{x} - 2L\log M}{\alpha\log\log{x}}}$ \enspace ,

\item $S_{\alpha,\epsilon}^{(+)}(x)$ be the set of integers $s$ which are the product of $k^{(+)}_\alpha(x)$ distinct elements from 
 \[ P_\alpha^{(+)}(x) \setminus P_{\alpha-\epsilon}^{(+)}(x) \enspace , \]
\item ${S_{\alpha,\epsilon}^{(-)}}\left(M_1(x),M_2(x),x\right)$ be the set of integers $s$ which are the product of the largest 
odd number not larger than $k^{(-)}_\alpha(x)$ many distinct elements from 
 \[ P_{\alpha}\left(M_1(x),M_2(x),x\right) \setminus P_{\alpha-\epsilon}\left(M_1(x),M_2(x),x\right) \enspace. \]
\end{itemize}
\end{nota}

The following two claims are immediate consequences of the construction.
\begin{claim}
The elements $s$ of $S_{\alpha,\epsilon}^{(+)}(x)$ all satisfy
   \[ \left((\log{x})^{-k^{(+)}_\alpha(x)\epsilon}\right)\frac{x^{1-o(1)}}{M^L} \leq s < \frac{x}{M^L} \]
   as $x\rightarrow \infty$.
\end{claim}

\begin{claim}
The elements $s$ of $S_{\alpha,\epsilon}^{(-)}(x)$ all satisfy
   \[ \left((\log{x})^{-k^{(-)}_\alpha(x)\epsilon}\right)\frac{x^{1-o(1)}}{M^L} \leq s < \frac{x}{M^{2L}} \]
      as $x\rightarrow \infty$.
\end{claim}

The next two propositions follow  from the lower bound on the size of $P_\alpha^{(\pm)}$ and the definition of $k_\alpha^{(\pm)}$.

\begin{proposition} \label{prop:Splussize}
If $\alpha$ satisfies the conditions of Proposition \ref{prop:lb1} then
      \[ \abs{S_{\alpha,\epsilon}^{(+)}(x) }\geq x^{1-\alpha^{-1} + o(1)} \]
         as $x\rightarrow \infty$.
\end{proposition}
\begin{proof}
A standard bound on a binomial coefficient is given by ${n \choose k} \geq (n/k)^k$.  We are choosing $k_{\alpha}^{(+)}$
many primes from a set of size at least $(\log{x})^{\alpha - o(1)} - (\log{x})^{\alpha - \epsilon} = (\log{x})^{\alpha - o(1)}$.
Noting that $M = x^{o(1)}$ by Proposition \ref{prop:Msize}, 
the resulting lower bound on $\abs{S_{\alpha,\epsilon}^{(+)}(x) }$ is
$$
\left( \frac{ (\log{x})^{\alpha - o(1)} }{ (\log{x})^{1+o(1)}} \right)^{\frac{ \log{x} - L\log M}{\alpha \log\log{x}} - 1}
\geq ( (\log{x})^{\alpha - 1 + o(1)})^{(\alpha^{-1} + o(1))\frac{\log{x}}{\log\log{x}}}
= x^{1-\alpha^{-1} + o(1)} \enspace .
$$
\end{proof}

\begin{proposition} \label{prop:Sminussize}
If $\alpha$, $M_1(x)$, and $M_2(x)$ satisfy the conditions of Proposition \ref{prop:lb2X} then
      \[ \abs{S_{\alpha,\epsilon}^{(-)}\left(M_1(x),M_2(x), x \right)} \geq x^{1-\alpha^{-1} + o(1)} \]
         as $x\rightarrow \infty$.
\end{proposition}
\begin{proof}
The proof is identical to that of Proposition \ref{prop:Splussize}.
\end{proof}

The next two propositions construct a composite $n$ with many degree-$2$ Frobenius liars.
The strategy in the plus one case is to start with a composite $s$ that is the product of many primes $p$ such that 
$p-1$ is smooth, then find a prime $q$ such that $n = sq$ is congruent 
to $1$ modulo $M$.  While the liar count primarily comes from the primes $p$ dividing $s$, we need 
to ensure at least one modulo $q$ liar, else the entire modulo $n$ liar count becomes $0$.

\begin{lemma}\label{lem:consQ1}
As before, let $L$ be an upper bound for Linnik's constant.
Given any element $s$ of $S_{\alpha,\epsilon}^{(+)}(x)$ there exists a prime $q < M^L$ such that
\begin{itemize}
\item $sq=1 \pmod{M}$,
\item $ \gcd(q,s) = 1$, and
\item $ \frac{1}{2}\left( \gcd(q-1,sq-1)^2- \gcd(q-1,sq-1)\right) > 0$.
\end{itemize}
Moreover, the number of liars of $n=sq$ with $\jac{\Delta}{n} = +1$ is at least $x^{2 - \epsilon \frac{2}{\alpha} - o(1)}$ as $x\rightarrow \infty$.
\end{lemma}
\begin{proof}
By construction, every $s \in S_{\alpha, \epsilon}^{(+)}(x)$ satisfies $\gcd(s,M) = 1$ since 
primes dividing $s$ are larger than $\log{x}$.
Then by the definition of $L$, we can choose $M < q < M^{L}$ to be the smallest prime such that 
$sq = 1 \pmod{M}$.  Since $q > M$ and the factors of $s$ are all smaller than $M$, we have 
$\gcd(q,s) = 1$.  With $q,n$ both odd, the third condition follows from Lemma \ref{lem:cong-cases}.

For a lower bound on $L_2^+(n)$ where $n = sq$ we count only the liars from primes $p \mid s$
with $\jac{\Delta}{p} = +1$.  This gives
$$
\prod_{p \mid s} L_2^{++}(n,p) = 
\prod_{p \mid s} \frac{1}{2} (\gcd(n-1,p-1)^2 - \gcd(n-1,p-1))
$$
by Lemma \ref{lem:L2++}.
By construction, for $p \mid s$ we have $p-1 \mid M$ and $M \mid n-1$, so the product becomes 
\begin{align*}
2^{-k^{(+)}_\alpha(x)}\prod_{p \mid s} (p-1)(p-2) 
&\geq 2^{-k^{(+)}_\alpha(x)} \cdot  s^{2-o(1)} \\
&\geq x^{-o(1)} \left((\log{x})^{-k_\alpha^{(+)}(x) \epsilon (2-o(1))} \right) \frac{x^{2-o(1)}}{M^{L (2-o(1))}} \\
& \geq x^{-o(1)} x^{-\epsilon \cdot \frac{2}{\alpha}(1+o(1))} \frac{x^{2-o(1)}}{x^{o(1)}}
= x^{2 - \epsilon \frac{2}{\alpha} - o(1)}
\end{align*}
where the upper bound on $M$ comes from Proposition \ref{prop:Msize}.
\end{proof}

In the minus one case we have two different divisibility conditions to satisfy, and as a result require two primes 
$q_1$ and $q_2$ to complete the composite number $n$.

\begin{lemma}\label{lem:consQ2}
Let $L$ be an upper bound for Linnik's constant.
Given any element $s$ of $S_{\alpha,\epsilon}^{(-)}(x)$ there exists a number $q < M^{2L}$ such that
\begin{itemize}
\item $sq=1 \pmod{M_1}$,
\item $sq=-1 \pmod{M_2}$,
\item $ \gcd(q,s) = 1$, and
\item $\prod_{p \mid q} \frac{1}{2}\left( \gcd((sq)^2-1,p-1)- \gcd(sq-1,p-1)\right) > 0.$
\end{itemize}
Moreover, the number of liars of $n=sq$ with $\jac{\Delta}{n} = -1$ is at least
\[ 2^{-k^{(-)}_\alpha(x)}\prod_{p \mid s} \left(p^2-1\right)   = x^{2-\epsilon \frac{2}{\alpha}-o(1)}  \]
   as $x\rightarrow \infty$.
\end{lemma}
\begin{proof}
We construct $q$ as the product of two primes $q_1$ and $q_2$.
Let $\ell_1,\ell_2$ be two distinct odd primes which divide $M_2$ and write $M_2 = M_2'\ell_1^{r_1}\ell_2^{r_2}$
where $\gcd(M_2', \ell_1 \ell_2)=1$.
Choose $q_1$ to be the smallest prime greater than $M$ satisfying the following four conditions:
$$
\begin{array}{ll}
sq_1 = 1 \pmod{M_1}  \hspace{1in} & sq_1 = -1 \pmod{M_2'} \\ q_1 = 1 \pmod{\ell_1^{r_1}}   \hspace{1in}& sq_1 = -1 \pmod{\ell_2^{r_2}} 
\end{array}
$$
and choose $q_2$ to be the smallest prime greater than $M$ satisfying the following four conditions:
$$
\begin{array}{ll}
 q_2 = 1\pmod{M_1} \hspace{1in} &q_2 = 1 \pmod{M_2'}  \\ sq_2 = -1 \pmod{\ell_1^{r_1}}  \hspace{1in} &q_2 = 1 \pmod{\ell_2^{r_2}} \enspace .
\end{array}
$$

Note that $q_1,q_2 > M$ implies they are greater than any factor of $s$, and thus relatively prime to $s$, 
and $\gcd(s,M) = 1$ since primes dividing $s$ are greater than $\log{x}$.
Then $q_1$, $q_2$ exist due to the definition of Linnik's constant, with $q_1, q_2 < (M_1 M_2' \ell_1^{r_1} \ell_2^{r_2})^L$
so that $q < M^{2L}$.  Note $s q_1 q_2 = 1 \pmod{M_1}$, which satisfies the first bulleted condition.  In addition, 
$sq_1 q_2 = -1 \pmod{M_2'}$, $sq_1 q_2 = -1 \pmod{\ell_1^{r_1}}$, and $sq_1 q_2 = -1 \pmod{\ell_2^{r_2}}$ so that 
$sq = -1 \pmod{M_2}$.  For the fourth bullet point, $sq_1 q_2 = -1 \pmod{\ell_1^{r_1}}$ and $q_1 = 1 \pmod{\ell_1^{r_1}}$
gives the result by Lemma \ref{lem:cong-cases}.

To bound $L_2^{-}(n)$ we select only $n$ where $\jac{\Delta}{p} = +1$ for all $p \mid q$
and $\jac{\Delta}{p} = -1$ for all $p \mid s$.  By Lemma \ref{Lem:L2--} we have 
\begin{align*}
\prod_{p \mid s} L_2^{--}(n,p) &= 
\prod_{p \mid s} \frac{1}{2}\left(\gcd(p^2-1,n^2-1,n-p)- \gcd(n-1,p-1)\right) \\
& = 2^{-k_{\alpha}^{(-)}(x)} \prod_{p \mid s} ( p^2-1) - (p-1)
\end{align*}
since by construction $p-1 \mid n-1$ and $p+1 \mid n+1$.  This product is $x^{2-\epsilon \frac{2}{\alpha} - o(1)}$
by the same argument as that in Lemma \ref{lem:consQ1}.
\end{proof}

\begin{proof}[Proof of Theorems \ref{thm:one} and \ref{thm:two}]
In each case the theorem is an immediate consequence of the lower bounds on the number of liars for each value of $n=sq$ constructed in Lemma \ref{lem:consQ1} or \ref{lem:consQ2}, together with the size of the set $S$ under consideration.

More specifically, for each element $s$ of $S_{\alpha,\epsilon}^{(\pm)}(x)$, by Lemma \ref{lem:consQ1} or \ref{lem:consQ2} 
we can associate a distinct number $n$ with 
 $L_2^\pm(n) \geq x^{2-\epsilon \frac{2}{\alpha}-o(1)}$.  For each of the plus, minus cases we have that
\[ \abs{ S_{\alpha,\epsilon}^{(\pm)}(x)} \geq x^{1-\alpha^{-1}-o(1)}  \]
for $\alpha$ satisfying as appropriate Proposition \ref{prop:lb1} or  \ref{prop:lb2}.
We conclude that for all  $\epsilon>0$ and appropriately chosen $\alpha$ we have
\[ \sum_n L_2^\pm(n) \geq x^{3-\alpha^{-1}-\epsilon \frac{2}{\alpha}-o(1)} \enspace . \]
Allowing $\epsilon$ to go to $0$, we obtain the result.
\end{proof}

\section{Upper bounds on the average number of degree-$2$ Frobenius pseudoprimes} \label{sec:upper}

Our proof will follow \cite[Theorem 2.2]{ErdosPomerance01} quite closely.  First we need 
a key lemma, the proof of which follows a paper of Pomerance \cite[Theorem 1]{Pomerance81}.

\begin{nota}
Given an integer $m$, define
\[ \lambda(m) = \underset{p \mid m}\lcm (p-1) \qquad\text{and}\qquad\lambda_2(m) = \underset{p \mid m}\lcm (p^2-1) \enspace . \]
Note that $\lambda(m)$ is not Carmichael's function, though it is equivalent when $m$ is squarefree.
Moreover, given $x>0$ we shall define
\[ \mathcal{L}(x) = {\rm exp}\left( \frac{\log{x} \log_3 x}{\log_2 x} \right) \]
where $\log_2 x = \log\log{x}$ and $\log_3 x = \log\log\log{x}$.  Here $\log$ is the natural logarithm.
\end{nota}

\begin{lemma} \label{lemma:lambda2count}
As $x\rightarrow \infty$ we have that 
$$
\# \{ m \leq x \ : \ \lambda_2(m)=n\} \leq x \cdot \mathcal{L}(x)^{-1 + o(1)} \enspace .
$$
\end{lemma}
\begin{proof}
For $c > 0$ we have 
$$
\sum_{m \leq x \atop \lambda_2(m)=n} 1 
\leq x^c \sum_{\lambda_2(m)=n} m^{-c}
\leq x^c \sum_{p \mid m \Rightarrow p^2-1 \mid n} m^{-c}
\leq x^c \sum_{p \mid m \Rightarrow p-1 \mid n} m^{-c} \enspace .
$$
By the theory of Euler products, we can rewrite the sum as $\prod_{p-1 \mid n} (1-p^{-c})^{-1}$.
Call this product $A$.  With $c = 1 - \frac{\log_3 x}{\log_2 x}$, the result follows if we 
can show that $\log{A} = o(\log{x}/\log_2 x)$.

Take $x$ large enough so that $\frac{\log_3 x}{\log_2 x} \leq \frac{1}{2}$; from this it follows that 
for all primes $p$, $\frac{1}{1-p^{-c}} \leq 4$.

Following Pomerance in \cite[Theorem 1]{Pomerance81}, via the Taylor series for $-\log(1-x)$ we can show that 
$$
\log{A} = \sum_{p-1 \mid n} \frac{p^{-c}}{1-p^{-c}}
\leq 4 \sum_{d \mid n} d^{-c} \leq 4 \prod_{p \mid n} (1-p^{-c})^{-1}
$$
and similarly
$$
\log\log{A} \leq \log 4 + \sum_{p \mid n} \frac{p^{-c}}{1-p^{-c}}
\leq \log 4 + 4 \sum_{p \mid n} p^{-c} \enspace .
$$
Since the sum is maximized with many small primes, 
an upper bound is 
$$
\log 4 + \sum_{p \leq 4 \log{x}} 4 p^{-c} = O\left( \frac{ (\log{x})^{1-c}}{(1-c) \log\log{x}} \right) 
$$
where the sum is evaluated using partial summation.
With $c = 1 - \log_3 x/\log_2 x$ we achieve
\[ \log\log{A} = O\left(\frac{\log_2 x}{\log_3 x}\right) \]
 so that $\log{A} = o(\log{x}/\log_2 x)$ as requested.
\end{proof}

An interesting question is whether the upper bound in that lemma can be lowered.  If so, a more clever upper bound 
would be required for the sum over primes $p$ dividing $m$ such that $p^2-1 \mid n$.

In \cite[Theorem 2.2]{ErdosPomerance01}
 the key idea is to parameterize composite $n$ according to the size of the subgroup 
of Fermat liars, and then to prove a useful divisibility relation involving $n$.  Here we reverse this strategy: 
we parameterize according to a divisibility condition and prove an upper bound on the size of the set of Frobenius liars.

\begin{lemma} \label{lemma:minus}
Assume $n$ is composite and let $k$ be the smallest integer such that $\lambda_2(n) \mid k(n^2-1)$.
Then
\[ L_2^-(n) \leq \frac{1}{k} \prod_{p \mid n} (p^2-1) \enspace .  \]
\end{lemma}
\begin{proof}
We have $k = \lambda_2(n)/\gcd(\lambda_2(n), n^2-1)$.  Our goal will be to show that 
\begin{equation} \label{division1}
\frac{\lambda_2(n)}{\gcd(\lambda_2(n), n^2-1)} \prod_{p \mid n} \gcd(p^2-1, n^2-1)
\left| \ \prod_{p \mid n} p^2-1 \right. \enspace .
\end{equation}
If this is true, then combined with Lemma \ref{lemma:liarcountupperbounds} we have
$$
L_2^-(n) \leq \prod_{p \mid n} \gcd(n^2-1, p^2-1) \leq \frac{1}{k} \prod_{p \mid n} p^2-1 \enspace .
$$
Fix arbitrary prime $q$  and let $q^{e_i}$ be the greatest power of $q$ that divides $p_i^2-1$.
Suppose we have ordered the $r$ primes dividing $n$ according to the quantity $e_i$.  Let $q^d$ be the power of $q$ that divides $n^2-1$.

Consider first the case where $d \geq e_r$, the largest of the $e_i$.  Then $q^{e_r}$ divides $\lambda_2(n)$ 
since it is defined as an $\lcm$ of the $p^2-1$, and $q^{e_r}$ divides $\gcd(\lambda_2(n), n^2-1)$ since 
$d \geq e_r$.  We are left with the observation that $\prod_{p \mid n} \gcd(p^2-1, n^2-1)$ is a divisor 
of $\prod_{p \mid n} p^2-1$, and thus in particular the $q$ power divides.

Next consider the case where $d \geq e_i$ for $i \leq \ell$ and $d < e_i$ for $i > \ell$.  Then $q^{e_r}$ divides $\lambda_2(n)$ 
since it is defined as an $\lcm$, and $q^d$ divides $\gcd(\lambda_2(n), n^2-1)$ since $d < e_r$.  The total power 
of $q$ dividing the LHS is then $e_r - d + (\sum_{i=1}^{\ell} e_i) + (r-\ell) d$.  We have 
\begin{align*}
\left(\sum_{i=1}^{\ell} e_i\right) + e_r-d + d(r-\ell)
 &= \left(\sum_{i=1}^{\ell} e_i\right) + d(r-\ell-1) + (d + e_r-d)\\
 &\leq \left(\sum_{i=1}^{\ell} e_i\right) + \left(\sum_{i=\ell+1}^{r-1} e_i \right) + e_r\\
 &= \sum_{i=1}^{r} e_i \enspace ,
\end{align*}
which is the power of $q$ dividing $\prod p^2-1$.
Since $q$ was arbitrary, (\ref{division1}) holds, which finishes the proof.
\end{proof}

The result for $L_2^+(n)$ is similar.  We do need a new piece of notation, namely given a prime $p$ we shall define
\[ d_n(p) = \begin{cases} (p-1)^2 & \text{if }\gcd(n-1, p-1)^2 > \gcd(n-1, p^2-1) \\ p^2-1 &\text{if }  \gcd(n-1, p-1)^2 \leq \gcd(n-1, p^2-1)
\enspace . \end{cases} \]

\begin{lemma} \label{lemma:plus}
Suppose $n$ is composite and let $k$ be the smallest integer such that  $\lambda(n) \mid k(n-1)$.
Then 
\[ L_2^+(n) \leq \frac{1}{k} \prod_{p \mid n} d_n(p) \enspace . \]
\end{lemma}
\begin{proof}
From Lemma \ref{lemma:liarcountupperbounds} we know that 
$$
L_2^+(n) \leq \prod_{p \mid n} \max( \gcd(n-1,p^2-1), \gcd(n-1,p-1)^2)
$$
and from the definition of $k$ we know that $k$ is exactly $\lambda(n)/\gcd(\lambda(n), n-1)$.
It thus suffices to show that
\begin{align}
\left. \frac{\lambda(n)}{\gcd(\lambda(n), n-1)} \prod_{p \mid n} \max(\gcd(n-1,p^2-1), \gcd(n-1, p-1)^2) \ 
\right| \prod_{p \mid n} d_n(p) \enspace . \label{eqn:gcd_division}
\end{align}

For an arbitrary prime $q$, let $q^{e_i}$ be the power of $q$ dividing $d_n(p)$ and let $q^d$ be the power of $q$ dividing $n-1$.
Order the $e_i$, and suppose that $d \geq e_i$ for $i \leq \ell$ and $d < e_i$ for $i > \ell$.
Then the exponent of $q$ dividing $\prod_{p \mid n} d_n(p)$ is $\sum_{i=1}^r e_i$.  Following the same argument 
as in Lemma \ref{lemma:minus}, the exponent of $q$ dividing the left hand side of (\ref{eqn:gcd_division}) is 
$$
(e_r - d) + \left( \sum_{i=1}^{\ell} e_i \right) + (r-\ell)d \leq \sum_{i=1}^r e_i \enspace .
$$
Since $q$ was arbitrary, the division in (\ref{eqn:gcd_division}) holds.
\end{proof}

\begin{theorem}
As $x\rightarrow \infty$ we have that
$$
\sum_{n \leq x}{}' L_2(n) \leq x^3 \mathcal{L}(x)^{-1 + o(1)}
$$
where $\sum'$ signifies the sum is only over composite integers.
\end{theorem}
\begin{proof}
Let $C_k(x)$ denote the set of composite $n \leq x$ where $k$ is the smallest integer such that $\lambda(n) \mid k(n-1)$, 
and let $D_k(x)$ denote the set of composite $n \leq x$ where $k$ is the smallest integer such that
 $\lambda_2(n) \mid k(n^2-1)$.  By Lemma \ref{lemma:minus}, if $n \in D_k(x)$ then $L_2^-(n) \leq n^2/k$.
 Similarly, by Lemma \ref{lemma:plus}, if $n \in C_k(x)$ then $L_2^+(n) \leq n^2/k$.
 Then
\begin{align*}
\sum_{n \leq x}{}' L_2(n)
&= \sum_{n \leq x}{}' L_2^+(n) + L_2^-(n) \\
&= \sum_{k}\sum_{n \in C_k(x)} L_2^+(n) + \sum_k \sum_{n \in D_k(x)} L_2^-(n) \\
&\leq \sum_{k} \sum_{n \in C_k(x)} \frac{n^2}{k} + \sum_k \sum_{n \in D_k(x)} \frac{n^2}{k} \\
& \leq \sum_{n \leq x} \frac{n^2}{\mathcal{L}(x)} + \sum_{k \leq \mathcal{L}(x)} \sum_{n \in C_k(x)} \frac{n^2}{k}
		+ \sum_{n \leq x} \frac{n^2}{\mathcal{L}(x)} + \sum_{k \leq \mathcal{L}(x)} \sum_{n \in D_k(x)} \frac{n^2}{k} \\
& = \frac{2x^3}{\mathcal{L}(x)} + x^2 \sum_{k \leq \mathcal{L}(x)} \frac{\abs{C_k(x)}}{k} + x^2 \sum_{k \leq \mathcal{L}(x)} \frac{\abs{D_k(x)}}{k}
\end{align*}
and thus the proof is complete if we can prove that $\abs{C_k(x)} \leq x \mathcal{L}(x)^{-1 + o(1)}$
and $\abs{D_k(x)} \leq x \mathcal{L}(x)^{-1 + o(1)}$ hold uniformly for $k \leq \mathcal{L}(x)$.

We focus first on the $D_k(x)$ result.  For every $n \in D_k(x)$, either 
\begin{enumerate}
\item[(1)] $n \leq x/\mathcal{L}(x)$, 
\item[(2)]
$n$ is divisible by some prime $p > \sqrt{k\mathcal{L}(x)}$, and/or 
\item[(3)] $n \geq x/\mathcal{L}(x)$ and $p \mid n$ implies $p \leq \sqrt{k\mathcal{L}(x)}$.
\end{enumerate}

The number of integers in case (1) is at most $x \mathcal{L}(x)^{-1}$ by assumption.  

Turning to case (2),
if $n \in D_k(x)$ and $p \mid n$ then $p^2-1$ is a divisor of $\lambda_2(n)$ and hence of $k(n^2-1)$.
This means that
\[ \left. \frac{p^2-1}{\gcd(k, p^2-1)} \right| n^2-1 \enspace .
\]  A straightforward application of the Chinese 
remainder theorem shows that the count of residues $x \pmod{a}$ with $x^2 = 1 \pmod{a}$ is at most 
$2^{\omega(a)+1}$ (for example, there are at most $4$ when working modulo $8$).  
Thus the count of $n \in D_k(x)$ with $p \mid n$ is at most
$$
\left \lceil \frac{2x 2^{\omega(p^2-1)}}{p (p^2-1)/\gcd(p^2-1, k)} \right \rceil
\leq \frac{2xk 2^{\omega(p^2-1)}}{p(p^2-1)}
= \frac{xk \mathcal{L}(x)^{o(1)}}{p (p^2-1)} \enspace .
$$
The equality $2^{\omega(p^2-1)+1} = \mathcal{L}(x)^{o(1)}$ follows from the fact that the maximum number of distinct prime 
factors dividing any integer $m \leq x^2$ is $(1 + o(1))\log(x^2)/\log\log(x^2)$ (see the proof 
of Proposition \ref{prop:lb2X} for the previous instance of this fact).
We conclude that the maximum number of $n$ in case $(2)$ is 
$$
\sum_{p > \sqrt{k \mathcal{L}(x)}} \frac{2xk \mathcal{L}(x)^{o(1)}}{p^3}
\leq xk \mathcal{L}(x)^{o(1)} \sum_{p > \sqrt{k \mathcal{L}(x)}} \frac{1}{p^3}
=x \mathcal{L}(x)^{-1 + o(1)} \enspace .
$$

For $n$ in case (3), since all primes dividing $n$ are small we know that $n$ has a divisor $d$ satisfying 
$$
\frac{x}{\mathcal{L}(x) \sqrt{k\mathcal{L}(x)}} < d \leq \frac{x}{\mathcal{L}(x)} \enspace .
$$
To construct such a divisor, remove primes from $n$ until the remaining integer is smaller than $x/\mathcal{L}(x)$; 
since each prime dividing $n$ is at most $\sqrt{k\mathcal{L}(x)}$ the lower bound follows.  
Let $A$ be the set of $d \in \Z$ that fall between the bounds given. We have 
$\lambda_2(d) \mid \lambda_2(n) \mid k(n^2-1)$, and so by a similar argument we know that the number of 
$n \in D_k(x)$ with $d \mid n$ is at most
$$
\frac{x \mathcal{L}(x)^{o(1)}}{d \lambda_2(d)/\gcd(k, \lambda_2(d))} \enspace .
$$
Unlike the case where $d$ is prime, here we might have $\gcd(d, \lambda_2(d)) \neq 1$.  But then the set of $n \in D_k(x)$
with $d \mid n$ is empty, so the bound given remains true.

Now, the number of $n \in D_k(x)$ in case (3) is at most
\begin{align*}
\sum_{d \in A} \frac{x \mathcal{L}(x)^{o(1)} \gcd(k, \lambda_2(d))}{d \lambda_2(d)}
&= x \mathcal{L}(x)^{o(1)} \sum_{d \in A} \frac{\gcd(k, \lambda_2(d))}{d \lambda_2(d)} \\
& = x \mathcal{L}(x)^{o(1)} \sum_{m \leq x} \frac{1}{m} \sum_{d \in A \atop \lambda_2(d)/\gcd(k, \lambda_2(d)) = m} \frac{1}{d} \\
&\leq x \mathcal{L}(x)^{o(1)} \sum_{m \leq x} \frac{1}{m} \sum_{u \mid k} \sum_{d \in A \atop \lambda_2(d) = mu} \frac{1}{d} \enspace .
\end{align*}
Note that if $\lambda_2(d)/\gcd(k, \lambda_2(d)) = m$, then $\lambda_2(d) = mu$ for some $u \mid k$, and thus 
summing over all $u \mid k$ gives an upper bound.

To evaluate the inner sum we use partial summation and Lemma \ref{lemma:lambda2count} to get 
\begin{align*}
\sum_{d \in A \atop \lambda_2(d) = mu} \frac{1}{d} 
&\leq \frac{1}{x/\mathcal{L}(x)} \sum_{d \in A \atop \lambda_2(d)=mu} 1 + 
	\int_{x/\mathcal{L}(x)\sqrt{k\mathcal{L}(x)}}^{x/\mathcal{L}(x)} \frac{1}{t^2} \sum_{d < t \atop \lambda_2(d) = mu} 1 \ {\rm d}t \\
& \leq \frac{\mathcal{L}(x)}{x} \frac{x/\mathcal{L}(x)}{\mathcal{L}(x/\mathcal{L}(x))^{1+o(1)}} + 
	\int_{x/\mathcal{L}(x)\sqrt{k\mathcal{L}(x)}}^{x/\mathcal{L}(x)} \frac{1}{t^2} \frac{t}{\mathcal{L}(t)^{1+o(1)}} \ {\rm d}t \\
& \leq \frac{1}{\mathcal{L}(x/\mathcal{L}(x))^{1+o(1)}} + \frac{\log{x}}{\mathcal{L}(x/\mathcal{L}(x)\sqrt{k\mathcal{L}(x)})}
= \mathcal{L}(x)^{-1+o(1)}
\end{align*}
for large enough $x$ and uniformly for $k \leq \mathcal{L}(x)$.  
Note that the count of divisors of an integer $k$ is bounded above by
$2^{(1+o(1))\log{k}/\log\log{k}}$ (see for instance \cite[Theorem 317]{HardyWright}).
Thus the count in case (3) is 
$$
\frac{x}{\mathcal{L}(x)^{1+o(1)}} \sum_{m \leq x} \frac{1}{m} \sum_{u \mid k} 1
\leq \frac{x \log{x}}{\mathcal{L}(x)^{1+o(1)}} 2^{(1+o(1))\frac{\log{k}}{\log\log{k}}}
= \frac{x}{\mathcal{L}(x)^{1+o(1)}}
$$
uniformly for $k \leq \mathcal{L}(x)$ and large enough $x$.

Proving $\abs{C_k(x)} \leq x \mathcal{L}(x)^{-1+o(1)}$ uniformly for $k \leq \mathcal{L}(x)$ will be similar.
Here the three cases are:
\begin{enumerate}
\item[(1)] $n \leq x/\mathcal{L}(x)$, 
\item[(2)] $n$ is divisible by some prime 
$p > k\mathcal{L}(x)$, and 
\item[(3)] $n \geq \mathcal{L}(x)$ and $p \mid n$ implies $p \leq k\mathcal{L}(x)$.
\end{enumerate}

If $n \in C_k(x)$ then $\lambda(n) \mid k(n-1)$.  Thus the number of $n \in C_k(x)$ 
with $p \mid n$ is at most 
$$
\left \lceil \frac{x}{p(p-1)/\gcd(p-1,k)} \right \rceil \leq \frac{xk}{p^2}
$$
and so the count of $n$ in case (2) is $x \mathcal{L}(x)^{-1 +o(1)}$.  

For $n$ in case (3) we know $n$ has a divisor $d$ satisfying 
$$
\frac{x}{k \mathcal{L}(x)^2} < d \leq \frac{x}{\mathcal{L}(x)}
$$
and so the bound of $x\mathcal{L}(x)^{-1+o(1)}$ follows exactly from 
case (3) of \cite[Theorem 2.2]{ErdosPomerance01}.
\end{proof}

\section{Conclusions and further work}

A very naive interpretation of Theorems \ref{thm:th1} and \ref{thm:th2} is that 
for any given $f$, you should expect that there are $n$ for which $f$ is a liar. Moreover, one expects to find this in both the $+1$ and $-1$ cases.
Likewise, one expects that given $n$, there will exist $f$ which is a liar in both the $+1$ and $-1$ cases.
To emphasize the extent to which one should be careful with the careless use of the word `expect' we remind the reader that in Section \ref{subsec:vanish} we describe infinite families of $n$ for which $L_2^-(n)=0$.
It would be interesting to know how often $L_2^-(n)$ vanishes for $n<x$.

It is useful to note that this vanishing described in Section \ref{subsec:vanish} gives some heuristic evidence to suggest that the Baillie-PSW test is significantly more accurate than other primality tests. Further work to make these heuristics more precise may be worth pursuing.

In contrast to the above, the proof of Theorem \ref{thm:th2} suggests that one should expect there to exist many Frobenius-Carmichael numbers (see \cite[Section 6]{Grantham01} for a definition) relative to quadratic fields $K$ for which $\jac{n}{\delta_K} = -1$.  
It is likely this heuristic can be extended to show that for each fixed quadratic field $K$ there exist infinitely many Frobenius-Carmichael numbers $n$ relative to $K$ with $\jac{n}{\delta_K} = -1$. 
A result of this form would be a nice extension of \cite{Grantham10}, 
which proved infinitely many Frobenius-Carmichael numbers $n$ for which $\jac{n}{\delta_K} = 1$. 
As such a number would also be a classical Carmichael number, such numbers would tend to lead to a failure of the Baillie-PSW test.
If this could be done for all $K$ it would show that all $f$ admit $n$ for which $f$ is a liar and the Jacobi symbol is $-1$.
It remains an open problem to prove such numbers exist.

It remains unclear from our results if the expected value of $L_2^-(n)$ is actually less (in an asymptotic sense) than the expected value of $L_2^+(n)$. 
Various heuristics suggest that it ought to be.
A result of this sort would put further weight behind the Baillie-PSW test.




\section{Acknowledgments}


The authors would like to thank Carl Pomerance for several comments which improved the results presented.




\providecommand{\bysame}{\leavevmode\hbox to3em{\hrulefill}\thinspace}
\providecommand{\MR}{\relax\ifhmode\unskip\space\fi MR }
\providecommand{\MRhref}[2]{%
  \href{http://www.ams.org/mathscinet-getitem?mr=#1}{#2}
}
\providecommand{\href}[2]{#2}

\end{document}